\documentclass[12pt]{amsart}

\usepackage{latexsym}
\usepackage{amsmath}
\usepackage{stmaryrd,epsfig}
\usepackage{xypic}

\usepackage{color}
\usepackage{amssymb}
\usepackage{amscd}
\usepackage{multicol}
\usepackage[cmtip,matrix,arrow]{xy}
\numberwithin{equation}{section}

\newtheorem{theorem}{\bf Theorem}[section]
\newtheorem{corollary}[theorem]{\bf Corollary}
\newtheorem{proposition}[theorem]{\bf Proposition}
\newtheorem{lemma}[theorem]{\bf Lemma}
\newtheorem{definition}[theorem]{\bf Definition}
\newtheorem{example}[theorem]{\bf Example}
\newtheorem{definition-theorem}[theorem]{\bf Theorem-Definition}
\newtheorem{remark}[theorem]{\bf Remark}

\def\bR{\mathbb{R}}
\def\bC{\mathbb{C}}
\def\bZ{\mathbb{Z}}
\def\bP{\mathbb{P}}

\def\t{\mathfrak{t}}
\def\g{\mathfrak{g}}
 \def\k{\mathfrak{k}}
  
\def\quott({/\! /}

\def\g{\mathfrak{g}}
\def\t{\mathfrak{t}}

\def\k{\mathfrak{k}}

\def\q{\mathfrak{q}}
\def\h{\mathfrak{h}}

\def\A{{\mathcal A}}

\title{Assignments for topological group actions}

\begin{document}


\vspace{0.5cm}

\maketitle

\begin{center}
{\small 
\begin{tabular}{ll }
{\sc Oliver Goertsches } & {\sc  Augustin-Liviu Mare} \\
Fachbereich Mathematik und Informatik $ \ \ $ & Dept.~of Mathematics and Statistics  \\
der Philipps-Universit\"at Marburg & University of
Regina \\
D-35032 Marburg, Germany & Regina SK, S4S 0A2 Canada\\
{\tt  goertsch@mathematik.uni-marburg.de} & {\tt mareal@math.uregina.ca}
 \end{tabular}
}

\end{center}

\vspace{0.3cm}

\begin{abstract} A polynomial assignment for a continuous action of a compact torus $T$ on a topological
space $X$ assigns to each $p\in X$ a polynomial function on the Lie algebra
of the isotropy group at $p$ in such a way that  a certain compatibility condition is satisfied.
The space $\A_T(X)$ of all polynomial assignments  has a natural  structure of an algebra over the polynomial ring
of ${\rm Lie}(T)$. It is an equivariant homotopy invariant, canonically related to the
equivariant cohomology algebra. In this paper we prove various properties of $\A_T(X)$ such as 
Borel localization, a Chang-Skjelbred lemma,  and a Goresky-Kottwitz-MacPherson presentation.
In the special case of Hamiltonian torus actions on symplectic manifolds we prove a
surjectivity criterion for the assignment equivariant Kirwan map
corresponding to a circle in $T$. We then obtain a Tolman-Weitsman type presentation  
 of the kernel of this map. 
 
${}$

 
 \end{abstract}


\section{Introduction}\label{sec:1}
The notion of assignment   associated to a torus action on a manifold 
was defined by Ginzburg, Guillemin, and Karshon in \cite{GGK}, by means of
a construction that takes into account exclusively the orbit stratification and the relative position of the strata.   
They were led to this construction while dealing with the existence problem of an abstract 
moment map for a given action. However, as the authors briefly mention, this new notion
is susceptible to be relevant for another important question in this area, namely,
under which circumstances is the equivariant cohomology algebra determined by the orbit stratification? 
 Indeed, a few years later,  Guillemin, Sabatini, and Zara 
have found in \cite{GSZ} a direct connection between  the equivariant cohomology and 
 a particular assignment
 space, which is called by  them the algebra of {\it polynomial assignments}.
Concretely, the connection is given by a ring homomorphism,
which, for certain classes of actions, is injective  and sometimes even bijective. 
For example, injectivity is achieved for equivariantly formal actions with isolated
fixed points on compact manifolds and bijectivity  for the sub-class of actions of Goresky-Kottwitz-MacPherson (GKM) type. 

This paper is based on the observation that the polynomial assignment algebra 
can be defined for arbitrary continuous (torus) actions on topological spaces. 
More precisely, let $X$ be a topological space and $T$ a torus
that acts on it. For any $p\in X$ we denote by $T_p$ the corresponding isotropy subgroup of $T$
and by $\t_p$ its Lie algebra (this will be referred to as the infinitesimal isotropy at $p$). 
  { Let also  $S(\t_p^*)$ be the algebra of polynomial functions on $\t_p$. 
\begin{definition}\label{defiggk} A {\rm (polynomial) assignment} for the $T$-action on $X$ is
a map $A$ that assigns to each $p\in X$ a polynomial $A(p)\in S(\t_p^*)$ such that for any (connected) subtorus $H\subset T$ the map $A^\h$ on the fixed point set $X^H$ is locally constant. Here $\h$ is the Lie algebra of $H$ and $A^\h$ the map defined by $A^\h(p) :=A(p)|_{\h}$,
for all $p\in X^H$.   
\end{definition}
}

This looks different from the definition in \cite{GSZ}
since, 
as already mentioned, the latter involves the orbit stratification.  
    However, we will show in   Section \ref{sec:smooth} below that for smooth actions on manifolds,
    there is no difference between the two notions.

We denote by $\A_T(X)$ the space of all assignments of the above type. It has an obvious canonical structure of an $S(\t^*)$-algebra, which will be referred to as the {\it assignment algebra} of the torus action. It defines a functor from the category of topological $T$-spaces 
to the category of $S(\t^*)$-algebras; moreover, it is an equivariant homotopy invariant,
see Section  \ref{basic} below.  Our goal here is to present some results concerning $\A_T(X)$
in the topological set-up. Direct connections with the equivariant cohomology algebra $H^*_T(X)$ in the spirit of \cite{GSZ} are also discussed, although they are not of main interest for us.
Polynomial assignments are studied here in their own right.

In fact,  equivariant cohomology is rather relevant for us in an indirect way:
 that is, we consider some results in this theory 
 and prove assignment versions of them.
 In the first part we will consider the inclusion of the 
 fixed point set $X^T$ into $X$ 
 along with the  map $\A_T(X)\to \A_T(X^T)$ induced by functoriality.
 After proving Borel type localization results, concerning the
 kernel and the cokernel of the aforementioned map, we 
 obtain an assignment version of the  GKM-theorem.
 It requires some extra assumptions on the
 (continuous) torus action. 
 Among others, we want the fixed point set $X^T$ to have only finitely
 many components, call them $Z_1,\ldots, Z_n$.
 Then the theorem says that $\A_T(X)$ is isomorphic to the subspace of 
 $S(\t^*) \times \cdots \times S(\t^*)$ ($n$ factors)
 consisting of tuples $(f_1,\ldots, f_n)$ with the property that 
 if $Z_i$ and $Z_j$ are contained in a connected component of some $X^H$, where
 $H\subset T$ is a codimension one subtorus of Lie algebra $\h$,
 then $f_i$ and $f_j$ are equal when restricted to $\h$.
 The precise statement can be found in Section \ref{subsec:gkm}.
 We emphasize that the result is purely topological. 
 One class of torus actions for which it holds true
 is the one of  equivariantly formal actions on compact Hausdorff spaces 
 with finitely many infinitesimal isotropies and finite dimensional cohomology. 
  { We note that this is in the spirit of \cite[Section 3​.4]{GGK}: a characterization of  
  ${\mathcal A}_T(X)$ similar to the one above is obtained there under the (more restrictive) hypotheses that 
  $X$ is a manifold, the $T$-action is smooth, and the compatibility relations are assumed for subtori  $H$  of arbitrary dimension.}

 In the last section of the paper we consider the particular context of Hamiltonian torus actions
 on compact symplectic manifolds, which are prominent examples of equivariantly formal,
 in general non-GKM, actions. 
More precisely, we study  the assignment version of the
 equivariant Kirwan map. Recall that the Kirwan map is a basic instrument 
 when dealing with the cohomology
  of symplectic quotients.   There is a rich literature devoted to this topic.
 We only mention here the seminal work \cite{Ki} of Kirwan and 
 the influential papers \cite{Gol} by Goldin  and \cite{To-We} by Tolman and Weitsman,
 which are directly related to our interests. To state our result, let us denote by 
 $T$ the torus which acts and by $M$ the compact symplectic manifold which is acted on.
 Let also $Q\subset T$ be { a  one-dimensional subtorus with Lie algebra $\q$}
  and $\Phi : M \to \q^*$ the moment map of the restricted $Q$-action.
  If $\mu \in \q^*$ is a regular value of $\Phi$,  the symplectic quotient 
  $M_0:=\Phi^{-1}(\mu)/Q$ is a symplectic orbifold with a  canonical action of $T/Q$.
  We first show that there is a well-defined {\it equivariant assignment Kirwan map}
  $\kappa^A:\A_T(M) \to \A_{T/Q}(M_0)$. In the case when $M_0$ is a manifold, 
  this was already noticed in \cite{GSZ}. 
   Unlike its cohomological counterpart, 
  $\kappa^A$ is in general not surjective. We prove the following surjectivity criterion. 
  Assume that for any connected component $F$ of $M^T$, the weights of the 
  isotropy representation along $F$ are linearly independent modulo collinearity
  (more precisely, after setting equal any two weights which are collinear, the resulting
  set must be linearly independent). Then $\kappa^A$ is surjective. 
  The proof uses ideas from  Morse theory for the moment map function $\Phi$, 
  which are applicable mainly due to the fact that $\A_T$ is a topological, equivariant homotopy invariant.   
  We also achieve   
  a description of the kernel of $\kappa^A$, which is the
  assignment version of  a result previously obtained in cohomology by Goldin  \cite{Gol}.  
  As a consequence, concrete formulas for $\A_{T/Q}(M_0)$ become available. 
  The details can be found in Section \ref{sec:kirwan}.

 {\bf Acknowledgements.} We wish to thank Silvia Sabatini for helpful comments.
 { We also thank the referee for carefully reading the manuscript and suggesting several improvements.}

\section{Basic facts}\label{basic}

We start with a list of examples. The importance of the first one is rather historical. 
It is only intended to remind that 
originally assignments arose from the study of moment maps, see \cite{GGK}. 

\begin{example}  {\rm Let { $(M,\omega)$} be a symplectic manifold endowed with a Hamiltonian action of a 
torus $T$. If $\Phi : M \to \t^*$ is a moment map, then $A(p):=\Phi(p)|_{\t_p}$, $p\in M$,
is a polynomial assignment. }
\end{example}

In the following situations the assignment algebra $\A_T(X)$ can be obtained by
direct calculations.  

\begin{example}\label{leer} {\rm  If  $X$  is the empty space acted on by a torus $T$, 
then there exists a unique map from  $X$  to  $S(\t^*)$.
This map does not change after multiplication with $0 \in S(\t^*)$.
Thus, $\A_T(\emptyset) = \{0\}$.}
\end{example}

\begin{example}\label{ex:trivial}  {\rm In the case when a torus $T$ acts trivially on a non-empty connected space $X$, then 
$\t_p=\t$ for all $p\in X$, hence any assignment is constant on $X$. This means that $\A_T(X) =S(\t^*)$. } 
\end{example}

\begin{example}\label{ex:weight} {\rm Let  $m, n>0$ be two integers. The weighted complex projective plane 
$\bP (m, n)$ is defined as the quotient $\left(\bC^2\setminus \{(0,0)\}\right)/_\sim$, where
$$(z_0, z_1) \sim (\lambda^m z_0, \lambda^n z_1), \quad (z_0, z_1) \in \bC^2\setminus \{(0,0)\}, \lambda \in S^1.$$ 
We consider the  action of $S^1$ on $\bP(m,n)$, given by
$$z.[z_0:z_1]:=[z^kz_0:z_1], \quad z\in S^1,$$
where $k>0$ is an integer.  The Lie algebra of $S^1$ is $\bR$.   The action has two fixed points, $p_1:=[1:0]$ and $p_2:=[0:1]$. The infinitesimal isotropy at any other point is $\{0\}$. 
In this case, one can identify the symmetric algebra  $S(\t^*)$  with the polynomial ring $\bR[u]$. 
The assignments for our action are of the form
$$A(p)=\begin{cases}
f_1, \  {\rm if } \ p =p_1,\\
f_2, \ {\rm if } \ p = p_2,\\
r, \  {\rm if} \ p \notin\{ p_1, p_2\},
\end{cases}
$$
where  $f_1, f_2\in \bR[u]$ and $r\in \bR$ are such that $f_1(0)=f_2(0)=r$.
}
\end{example}

\begin{example}\label{modif} {\rm  Let $n\ge 1$ be an integer number and  $\xi_1, \ldots, \xi_n\in \bC$ the $n$-th roots of unity.
   Let $X$ be the quotient space $S^1 \times S^1 / \{(\xi_i, z) \sim (\xi_i, z') \ {\rm for \ all \ } i=1, \ldots, n \ 
   {\rm and \ all \ } z, z' \in S^1\}$. The action of $S^1$ on 
   $S^1 \times S^1$ given by $z.(x_1, x_2):=(x_1, zx_2)$ descends to $X$ and has $n$ fixed points, 
   $p_1, \ldots, p_n$, which are represented by $\{\xi_1\}\times S^1, \ldots, \{\xi_n\}\times S^1$, respectively.
   At any $p \in X$ which is not fixed by $S^1$, the isotropy is $\{1\}$. Hence the assignments are this time 
   $$A(p)=\begin{cases}
f_i, \  {\rm if } \ p =p_i \ {\rm for \ some \ } 1\le i \le n\\
r, \  {\rm if} \ p \neq p_i  \ {\rm for \ all \ } 1\le i \le n,
\end{cases}
$$ where $f_1,\ldots, f_n \in \bR[u]$ and $r\in \bR$ are such that $f_1(0) =\ldots = f_n(0) = r$. This is a slightly modified version of an example from
\cite{Ty} (see Figure 4 in that paper).}

\end{example}

\begin{example}\label{sphere} {\rm The torus $T:=S^1\times S^1$ acts on the sphere
$S^3:=\{(x_1, x_2)\in \bC^2 \mid |x_1|^2+|x_2|^2=1\}$ as follows:
$$(z_1, z_2).(x_1, x_2):=(z_1x_1, z_2 x_2).$$
It is an easy exercise to show that, after identifying $S(\t^*)=\bR[u_1, u_2]$, the assignments of this action are of the following form:
$$A((x_1,x_2))=\begin{cases}
f_1, \  {\rm if} \ x_1 =0,\\
f_2, \  {\rm if } \ x_2 =0,\\
r, \ {\rm if} \ x_1 \neq 0 \ {\rm and} \ x_2 \neq 0,
\end{cases}
$$ 
where $f_1, f_2 \in \bR[u]$ and $r\in \bR$ are such that $f_1(0)=f_2(0)=r$.}
\end{example}


In what follows we will show that the assignment algebra of a torus action
 shares with equivariant cohomology two important properties:
homotopy invariance and (a weak version of)
the Mayer-Vietoris sequence.

A $T$-equivariant map $f: X\to X'$ between two $T$-spaces $X$ and $X'$ induces an $S(\t^*)$-algebra homomorphism $f^*:\A_T(X')\to \A_T(X)$ between the corresponding assignment algebras: $$f^*(A)(p):= A(f(p))|_{\t_p},
\ {\rm  for \  all \ }  \ p\in X.$$ 
The map $f^*(A)$ satisfies Definition \ref{defiggk}.   
To verify this, take $H\subset T$ and $Y$ a component of $X^H$. 
Then $f(Y)$ is contained in a component of $(X')^H$, hence $A(f(p))|_\h$ is independent of $p \in Y$.

\begin{proposition}\label{homotopic} Let $f,g:X \to X'$ be two $T$-equivariant maps which are homotopic to each other through $T$-equivariant maps. Then $f^*=g^*: \A_T(X')\to \A_T(X)$.
\end{proposition}
\begin{proof}
Let $F:[0,1]\times X\to X'$ be a $T$-equivariant homotopy from $f$ to $g$. For every assignment $A\in \A_T(X')$ and every point $p\in X$ we have to show that $f^*(A)(p)=g^*(A)(p)$, i.e., that $A(f(p))|_{\t_p} = A(g(p))|_{\t_p}$. But by equivariance, $T_p\subset T_{F(t,p)}$ for all $t\in [0,1]$, hence the curve $t\mapsto F(t,p)$ lies completely in a connected component of $(X')^H$, where $H$ is the identity component of $T_p$. The claim thus follows 
from  { Definition \ref{defiggk}}.
\end{proof}

If $X$ is a $T$-space and $Y\subset X$ a $T$-invariant subspace, there is a {\it restriction map}
$\A_T(X)\to \A_T(Y)$, $A\mapsto A|_Y$,  which is the algebra homomorphism  induced by  the inclusion $Y\hookrightarrow X$.  

\begin{proposition}\label{mv} Let $X$ be a $T$-space and $Y, Z\subset X$ two $T$-invariant 
 subspaces, which are either both open or both closed. 
Then the following sequence is exact:
$$0\longrightarrow \A_T(Y\cup Z) \longrightarrow \A_T(Y)\oplus \A_T(Z) \longrightarrow \A_T(Y\cap Z).$$
Here $A_T(Y\cup Z) \to \A_T(Y)\oplus \A_T(Z)$ is given by $A\mapsto (A|_Y, A|_Z)$ and
$\A_T(Y)\oplus \A_T(Z) \to \A_T(Y\cap Z)$ by $(A,B)\mapsto A|_{Y\cap Z}-B|_{Y\cap Z}$. 
\end{proposition}  

\begin{proof}
The only nontrivial statement is that if $B\in \A_T(Y)$ and $C\in \A_T(Z)$ are such that
$B|_{Y\cap Z}=C|_{Y\cap Z}$ then there exists $A\in \A_T(Y\cup Z)$ such that
$A|_Y=B$ and $A|_Z=C$. Observe first that the last two equations define $A$ uniquely, as a map
$Y\cup Z \ni p \mapsto A(p)\in S(\t^*_p)$. It remains to show that $A$ is an assignment on $Y\cup Z$: that is, if $H\subset T$ is a subtorus  then $A^\h$ is
locally constant on $(Y\cup Z)^H=Y^H\cup Z^H$.

{\it Case 1.} Both $Y$ and $Z$ are open in $X$. Take $p\in Y^H \cup Z^H$, 
say $p\in Y^H$. Then $B^\h$ is constant on an open neighborhood of $p$ in $Y^H$,
which is also a neighborhood of $p$ in $Y^H\cup Z^H$, since $Y^H$ is open in the latter union. 
Finally, by definition, $A^\h$ and $B^\h$ coincide on that neighborhood. 

{\it Case 2.} Both $Y$ and $Z$ are closed in $X$. Again we take $p\in Y^H \cup Z^H$.
If $p \in Y^H \setminus Z^H$, there exists an open neighborhood $U$ of $p$ in $X$
such that $B^\h$ is constant on $U\cap Y^H$ and $U\cap Z^H=\emptyset$. 
But then $U\cap (Y^H\cup Z^H) = U\cap Y^H$ and on this set $B^\h$ equals $A^\h$,
hence the latter is constant as well. 
If $p\in Y^H\cap Z^H$, there exist open neighbourhoods $U_1$ and $U_2$ of $p$ in $X$
such that $B^\h$ is constant on $U_1\cap Y^H$ and $C^\h$ is constant on $U_2\cap Z^H$.
It follows that $A^\h$ is constant on $(U_1\cap U_2)\cap(Y^H\cup Z^H)$.   
\end{proof}

\section{Assignments for smooth group actions}\label{sec:smooth}

In this section $X$ is assumed to be a manifold and the $T$-action smooth. 
Polynomial assignments for such actions have been defined in \cite{GSZ}  as follows (cf.~also \cite[Definition 3.7]{GGK}).
One first considers the corresponding infinitesimal orbit-type stratification of $X$. 
That is, the strata are connected components of spaces of the form  $Y:=\{p\in X \mid \t_p=\h\}$,
where $\h$ is an infinitesimal isotropy. Let us denote $\t_Y:=\t_p$, where $p\in Y$.
There is a partial order $\preceq$ on the set of all strata given by 
$Y\preceq Z$ if and only if $Y\subset \overline{Z}$. Note that the last condition implies
$\t_Z \subset \t_Y$.  

\begin{definition} {\rm (\cite[Definition 2.1]{GSZ})} \label{defigsz} A {\rm polynomial assignment} for the $T$-action on $X$ is a function $A$ that associates
to each infinitesimal stratum $Y$ a polynomial $A(Y)\in S(\t_Y^*)$ such that if $Y\preceq Z$ then $A(Z)=A(Y)|_{\t_Z}$.    
\end{definition}

The following proposition says that  Definitions \ref{defiggk} and \ref{defigsz} are equivalent.

\begin{proposition}\label{equivdefi} (a) If $A$ is like in Definition \ref{defiggk} then $A$ is constant on each infinitesimal 
stratum $Y$. Let $A'(Y)$ denote the common value of all $A(p)$, $p\in Y$. Then the map $Y\mapsto A'(Y)$ satisfies
the requirement of Definition \ref{defigsz}. 

(b) If $A'$ is like in Definition \ref{defigsz} then the map $A$ given by $A(p):=A'(Y)$, where $p$ is in the stratum $Y$,
satisfies the { requirement} of Definition \ref{defiggk}.  
\end{proposition}
 
\begin{proof} (a) If $Y$ is a stratum, then obviously $\overline{Y}\subset X^{T_Y}$, where $T_Y\subset T$ is the connected Lie subgroup corresponding to $\t_Y$.  
This implies the claim.

(b)  Let $H\subset T$ be a subtorus and $p\in X^H$. 
Let $U$ be a tubular neighborhood  around $Tp$, i.e., an open neighborhood of $Tp$ which is $T$-equivariantly diffeomorphic to
$T\times_{T_p}\nu_p$, where $\nu_p$ is the normal space to $Tp$ at $p$. 

We claim that $A(q)|_\h$ is independent of  $q\in U\cap X^H$ (this will imply that
$A^\h$ is locally constant on $X^H$).
To prove this, we may assume that $q\in \nu_p$  and $q\neq 0$.
Since the $T_p$-action on $\nu_p$ is linear, the infinitesimal stratum of $q$ contains the half-line { $\{xq\mid x>0\}$} in $\nu_p$.
Denote the  stratum of $q$ by $Y$.
{ Since the half-line mentioned above is contained in $Y$,} we deduce  that $p\in \overline{Y}$. Hence the whole stratum of $p$ is contained in $\overline{Y}$,
thus  $A(q)=A(p)|_{\t_q}$.
But $q\in X^H$, which implies that $\h\subset \t_q\subset \t_p$ and further that $A(q)|_\h=A(p)|_\h$.
\end{proof}

\section{A Borel type localization theorem}

 Let $X$ be a connected topological space acted on by a torus $T$.
We assume throughout this section that the following assumption is  fulfilled.

\noindent{\bf Assumption 1.} The $T$-action on $X$ has only finitely many infinitesimal isotropies.

Recall that $\A_T(X)$ has a canonical structure of an $S(\t^*)$-algebra. 
The goal here is to prove an analogue of  Borel's localization theorem for equivariant cohomology,
see for instance \cite[Theorem C.20]{GGK1}.

Recall from Section \ref{basic} that if $Y$ is a $T$-invariant subspace of $X$, then there is a natural restriction map $\A_T(X)\to \A_T(Y)$.

\begin{proposition}\label{ttx} If the $T$-action on $X$ satisfies Assumption 1, then  the kernel of the
restriction map $r:\A_T(X)\to \A_T(X^T)$ is the $S(\t^*)$-torsion submodule of $\A_T(X)$.\end{proposition}

\begin{proof} 
We only need to show that any element in the kernel is $S(\t^*)$-torsion, since the other inclusion is obvious (note that $\A_T(X^T)$ is a free $S(\t^*)$-module, see Example 
\ref{ex:trivial}). Take $A\in \A_T(X)$ such that $r(A)=0$. Let  $\h_1,\ldots,\h_n$ be the infinitesimal isotropies of the $T$-action which are
different from $\t$. Pick a non-zero polynomial $f$ in $S(\t^*)$ 
which vanishes on $\h_1\cup \cdots \cup \h_n$. If one multiplies by $f$ any element of
$\A_T(X\setminus X^T)$ one obtains zero. Consider now the map $\A_T(X) \to \A_T(X^T)\oplus \A_T(X\setminus X^T)$,
which is the direct sum of the restriction maps. This map is  injective and maps $fA$ to zero. Thus, $fA=0$.
\end{proof}


\begin{corollary}\label{t-act} Assume that the  $T$-action on $X$ satisfies Assumption 1. Then the following assertions are equivalent:
\begin{itemize}
\item[(i)]  The set $X^T$ is not empty and the  restriction map $r: \A_T(X)\to \A_T(X^T)$ is injective.
\item[(ii)] The assignment algebra $\A_T(X)$ is $S(\t^*)$-torsion free.
\end{itemize}
\end{corollary}

\begin{proof} The implication ${\rm (ii)} \Rightarrow {\rm (i)}$   follows immediately from Proposition \ref{ttx}, see also Example \ref{leer}.
It remains to justify ${\rm (i)}\Rightarrow {\rm (ii)}$. But if $X^T\neq \emptyset$ and $r$ is injective, then  $\A_T(X)$ is a submodule of $\A_T(X^T)$; since the latter is free, the former is  torsion free.  
\end{proof}

 A class of examples which satisfy the two conditions in the corollary consists of smooth torus actions on compact smooth manifolds 
that are equivariantly formal (see Section \ref{cef} and Proposition \ref{propres} below). 

If $T$ is a circle and the two conditions in the corollary hold true,   then 
$\A_T(X)$ is not only torsion free, but also free, because it is a submodule of $\A_T(X^T)$, which is free,
and $S(\t^*)$ is a PID. In general, however, it is possible that $\A_T(X)$ is torsion free but not free: see Example \ref{ex:notfree}.


In the spirit of Borel's localization theorem for equivariant cohomology, not only the kernel of
$r:\A_T(X)\to \A_T(X^T)$ is $S(\t^*)$-torsion, but also its cokernel:

\begin{proposition}\label{coker} Consider an action of a torus $T$ on a topological space $X$ satisfying Assumption 1. Then the cokernel of the map $r:\A_T(X)\to \A_T(X^T)$ is $S(\t^*)$-torsion.
\end{proposition}
\begin{proof}
We have to show that for every assignment $A\in \A_T(X^T)$ there exists a polynomial $f\in S(\t^*)$ such that $fA$ is in the image of $r$.

Let $f$ be any polynomial that vanishes on all proper (i.e., $\neq \t$) infinitesimal isotropies of the action. We define an assignment $B$ on $X$ by declaring 
\[
B(p) = \begin{cases} fA(p) \in S(\t^*), & {\rm if} \ p\in X^T \\ 0, & {\rm if } \ p\notin X^T. \end{cases}
\] 
This really defines an assignment. { Let us check that it satisfies the condition in Definition \ref{defiggk}.}
 If $\h$ is contained in a proper infinitesimal isotropy, then $B^\h$ is identically zero; otherwise, $X^H=X^T$ and for any $p$ in a connected component of this space one has $B^\h(p)=f|_{\h}A(p)|_{\h}$, which  
is constant on that component, since $A$ is an assignment on $X^T$.  
\end{proof}

\begin{corollary}\label{cor-rank} For an action of a torus $T$ on a topological space satisfying Assumption 1, the restriction map 
$r:\A_T(X)\to \A_T(X^T)$ is an isomorphism modulo torsion.
Consequently, the rank of $\A_T(X)$ over ${S(\t^*)}$ is equal to the number of connected components of $X^T$.
\end{corollary}

\section{A Chang-Skjelbred  lemma}\label{sec:cs}

Let  $X$ be a  connected topological space acted on by the torus $T$. Consider 
the {\it 1-skeleton} of the action, which is
$X_1:=\{p\in X \mid {\rm corank } \ T_p \le 1\}.$ 

Besides Assumption 1 in the previous section, the following extra condition is needed here:

\noindent{\bf Assumption 2.} For any  subtorus $H\subset T$,  any component of $X^H$ has non-trivial intersection with $X^T$ and connected intersection with $X_1$.

 \subsection{Example: equivariantly formal actions}\label{cef}
In this subsection we will show that Assumption 2 is fulfilled  if $X$ is compact Hausdorff and the  $T$-action is 
 equivariantly formal  in the sense that $H^*_T(X)$ is  free relative to its canonical structure of $H^*(BT)$-module. 
 (We consider here \v{C}ech cohomology with real coefficients.) 
We first prove the following lemma.

\begin{lemma}\label{eqfo} Assume that  $X$ is compact Hausdorff, the $T$-action on $X$ is equivariantly formal
 and $\dim H^*(X)<\infty$. Then for any subtorus $H\subset T$, the $T$-action on (any connected component of) $X^H$ is equivariantly formal.
\end{lemma}

\begin{proof} { Recall that if $Y$ is any compact Hausdorff space acted on continuously by
the torus $T$, then $\dim H^*(Y^T)\le \dim H^*(Y)$, with equality if and only if the $T$-action is
equivariantly formal (see e.g.~\cite[Corollary 2, p.~46]{Hs}).} In particular, the claim in the lemma is equivalent to $\dim H^*((X^H)^T) = \dim H^*(X^H)$. 
But $(X^H)^T=X^T$, and hence $$\dim H^*(X^T)\le \dim H^*(X^H)\le \dim H^*(X).$$
By the aforementioned general result, we have $\dim H^*(X^T)=\dim H^*(X)$, hence the inequalities above are both equalities.
\end{proof}

From \cite[Corollary 1, p.~45]{Hs} we deduce that under the assumptions in the lemma, any connected component of $X^H$ contains a $T$-fixed point. Furthermore, the 1-skeleton of that component is
connected: this follows from \cite[Proposition 2.5]{CS}.
(Both these results are known   in the particular context of  differentiable group actions on manifolds: see, e.g.,
\cite[Theorem 11.6.1]{GS} or \cite[Lemma 3.1]{Go-Ma}.) 

\begin{example}
{\rm Not every action that satisfies Assumptions 1 and 2 is equivariantly formal. 
To see this, consider again Example  \ref{modif}: as noticed in \cite{Ty}, the action of $S^1$ on $X$ is not equivariantly formal;
the reason is that it has finitely many fixed points and $H^1(X)= \bR$.   
Another example is obtained by taking $S$, the (circular) subgroup of ${\rm SU}(3)$ which consists of all diagonal matrices of the form
 ${\rm Diag}(z^{-1}, z^{-2},z^3)$, where $z\in S^1$. The action of $S$  on ${\rm SU}(3)/S$
 given by multiplication from the left 
 is not equivariantly formal, see \cite[Proposition 8.9]{Ca}. 
 It satisfies Assumptions 1 and 2 in an obvious way. If one wants an action  with a non-trivial 1-skeleton, take the direct product of the
 $S$-action above with itself:  $S\times S$ on
  ${\rm SU}(3)/S \times {\rm SU}(3)/S$. One can easily check that
  this is again not equivariantly formal and
  satisfies Assumption 1.  Assumption 2 only needs to be checked for
  $H=\{(I,I)\}$, $H=\{I\}\times S$, and $H=S\times \{I\}$ and this can be done immediately
  (note that  the 1-skeleton is
  the union 
   $\left[({\rm SU}(3)/S)^S\times {\rm SU}(3)/S \right]\cup 
   \left[{\rm SU}(3)/S\times ({\rm SU}(3)/S)^S\right]$,
   which is a connected subspace of ${\rm SU}(3)/S\times {\rm SU}(3)/S$).  }
\end{example}

\subsection{The Chang-Skjelbred ``lemma"}
 
We will prove an assignment version of  \cite[Lemma 2.3]{CS}.

\begin{proposition}\label{propres} If Assumptions 1 and 2 are valid, then the restriction map
$r:\A_T(X) \to \A_T(X^T)$ is injective. Its image is the same as the image of
$r':\A_T(X_1) \to \A_T(X^T)$.
\end{proposition}

\begin{proof} We first show that $r$ is injective. Take $A\in \A_T(X)$ such that $r(A)=0$. Take $p\in X$ arbitrary. Denote by $H$ the identity component of $T_p$
 and by $\h$ its Lie algebra. 
By Assumption 2, the connected component of $X^H$ through $p$ contains a connected component of $X^T$.
 On the latter component $A$ is identically zero, hence $A^\h$ is identically zero on the former component as well. 
 This implies that $A(p)=0$.

For the second claim in the proposition, observe that one can factorize $r$ as
$$\A_T(X) \to \A_T(X_1) \stackrel{r'}{\to} \A_T(X^T).$$
Hence the image of $r$ is contained in the image of $r'$. 
We now prove the other inclusion.  

We consider $A\in \A_T(X_1)$ and construct $B\in \A_T(X)$ such that $r(B)=r'(A)$.
It will be convenient to use the following notation: if $Z\subset X$ is a connected component of $X^T$,
then
$$A(Z):=A(z), \ {\rm for \ all \ } z \in Z.$$
Take $p\in X$, let $H$ be the identity component of $T_p$, set $\h:={\rm Lie}(H)$, and denote by $Y$ the connected component 
of $X^H$ that contains $p$. 
 By Assumption 2, there exists a
component $Z$ of $X^T$ such that $Z\subset Y$. We set
$$B(p):=A(Z)|_{\h}.$$
We show that $B$ is well defined (i.e., $B(p)$ is independent of the choice of $Z$) and an assignment as well.
To this end, we take a  subtorus
$G \subset T$, $G\neq T$, a connected component $Y$ of $X^G$, two components $Z$ and
$Z'$ of $X^T$, both contained in $Y$, and show that
\begin{equation}\label{ay}A(Z)|_{\g} = A(Z')|_{\g},\end{equation}
where $\g:={\rm Lie}(G)$.
Indeed, by Assumption 2,  $Y\cap X_1$ is a connected subspace of 
$X_1^G$.  Hence the function $p\mapsto A(p)|_\g$ is constant on $Y\cap X_1$. 
Also note that $Z$ and $Z'$ are both contained in $Y\cap X_1$.   
\end{proof}

\subsection{A GKM description of the assignment algebra}\label{subsec:gkm}
 { Assumption 1 will be valid throughout this section.} Let us consider all possible infinitesimal isotropies  
which have codimension one in $\t$; say that they  are $\g_1,\ldots, \g_m$. 
By definition, the 1-skeleton of the $T$-action on $X$  is the union  of all  $X^{\g_i}:=\{p\in X \mid \g_i\subset \t_p\} $, $1\le i \le m$. 
The following supplementary requirement will be needed in this subsection:

\noindent{\bf Assumption 3.} Each of the spaces $X^\t$ and $X^{\g_i}$, $1\le i \le m$, has finitely many connected components. 

For example, this condition is satisfied if $X$ is compact Hausdorff with $\dim H^*(X) <\infty$, see, e.g.~\cite[Corollary 3.10.2]{AP}.

Assumptions 1, 2, and 3  alone lead to a presentation of the algebra $\A_T(X)$ which is similar to the one given by Goresky, Kottwitz, and MacPherson \cite{GKM} for the equivariant cohomology algebra $H_T^*(X)$. Recall that the latter presentation requires several other assumptions:
the $T$-action on $X$ must be equivariantly formal,  $X^T$ must be finite, and $X_1$ must be a union  of 2-spheres  (cf.~also  \cite[Section 11.8]{GS}). 
 
We denote by $Z_1, \ldots, Z_n$  the connected components of $X^T$. 

\begin{theorem}\label{gkmassign} 
If Assumptions 1, 2, and 3 are satisfied, then the image of $\A_T(X)$ under the injective algebra
homomorphism 
$r:\A_T(X)\to \A_T(X^T)$ is the subalgebra of $\A_T(X^T)=\oplus_{r=1}^n\A_T(Z_r)=S(\t^*)^n$ which consists of all $(f_1, \ldots, f_n)$ 
with the property that whenever $Z_r$ and $Z_s$ are contained in the same component of some $X^{\g_i}$, $1\le i \le m$, the difference $f_r-f_s$ restricted to  $\g_i$ is identically zero.   
\end{theorem}

\begin{proof}
If $A\in \A_T(X)$ then its restriction to $X^T$ is an $n$-tuple $(f_1, \ldots, f_n)$ which obviously satisfies the conditions in the lemma. To prove the other inclusion, we start with $(f_1,\ldots, f_n)$ with the properties in the
lemma. Consider the map $A_1$ on $X_1$ given by
$$
A_1(p) =   
\begin{cases}
f_r, {\rm \  if \ } p \in Z_r\\
f_r|_{\g_i}, {\rm \ if \ } \t_p=\g_i \ {\rm and } \ Z_r \subset X^{\g_i}.
\end{cases}
$$
Note that, by hypothesis, the polynomial $f_r|_{\g_i}$ does not depend
on $r$ with $Z_r \subset X^{\g_i}$. We now show that $A_1$ is an assignment on $X_1$.
Let $H\subset T$ be a subtorus. If $H=T$ then $X_1^H=X^T$ and $A_1$ is obviously
constant on each component of the latter space.  If $H\neq T$, we have 
$$X_1^H = \bigcup_{\h \subset \g_i} X^{\g_i}.$$
By Assumption 2, this is a connected space. We need to show that $A_1^\h$ is constant
on this space.  Look at the connected components of the spaces $X^{\g_i}$ for all
$i\in \{1,\ldots, m\}$ with $\h \subset \g_i$. The intersection of two such subspaces is empty or is a union of one 
or more $Z_r$. Since $X_1^H$ is connected, for any of the two aforementioned  components, say $Y$ and $Y'$, there exists a chain 
of components, $Y_1, \ldots, Y_q$, such that $Y_1=Y$, $Y_q=Y'$ and 
$Y_i\cap Y_{i+1}\neq \emptyset$ for all $1\le i \le q-1$.
But $A_1^\h$ is constant on each $Y_i$, hence the values on $Y$ and $Y'$ are equal. 
Thus $A_1$ is an assignment on $X_1$. 
Finally, by Proposition \ref{propres},  $A_1$  can be extended from $X_1$ to an assignment on $X$. 
 \end{proof}

The $S^1$-action described in Example \ref{modif} satisfies the hypotheses of the theorem.
We have already calculated $\A_{S^1}(X)$ and the  presentation we found is of GKM type. 
It is worthwhile noting that no such presentation exists for $H^*_{S^1}(X)$.
The reason is that, as  already mentioned, the action has finitely many fixed points and $H^1(X)=\bR$.

\section{Relation to equivariant cohomology}\label{sec:relcoh}

An important class  of assignments arise from equivariant cohomology.
Let $X$ be a compact Hausdorff topological space.
Recall that $H^*_T(X)=H^*(E\times_T X)$, where
$E=ET$ is the total space of the classifying bundle of $T$.  
(By $H^*$ we mean \v{C}ech cohomology with real coefficients.)   
We will use the identification 
\begin{equation}\label{iden}H^*_T(T/H)=S(\h^*),\end{equation} for any connected and closed subgroup $H\subset T$. Concretely, $$H^*_T(T/H)=H^*(E\times_T (T/H)) =H^*(E/H) =H^*(BH)=S(\h^*).$$
We now define  $\gamma_X : H^*_T(X) \to \A_T(X)$ as follows: to $\alpha \in H^*_T(X)$ corresponds the assignment $A$ given by
\begin{equation}\label{assign?}A(p) :=\alpha|_{Tp},\quad p\in X.\end{equation}
The right hand side represents the image of $\alpha$ under the map $i_p^*:H^*_T(X)\to H^*_T(Tp)$ induced by the inclusion $i_p:Tp\hookrightarrow X$ (the identification (\ref{iden}) is taken into account).
\begin{proposition} The map $A$ defined by (\ref{assign?}) is an assignment.
\end{proposition}
\begin{proof} 
We need to show that $A$ satisfies the condition in Definition \ref{defiggk}.  Let $H\subset T$ be a connected  and closed subgroup and $Y\subset X$ a connected component of $X^H$. For $p\in Y$, the inclusion $i_p:Tp \hookrightarrow X$ factorizes as
$Tp \hookrightarrow Y \hookrightarrow X$. Moreover, the map $S(\t_p^*)\to S(\h^*)$ given by restriction to $\h$
is actually the same as $H^*_T(Tp) \to H^*_T(T/H)$ induced by  $a_p:T/H \to T p\subset Y$, $tH \mapsto tp$, for all $t\in T$ (indeed, this is the only $S(\t^*)$-algebra homomorphism  $S(\t_p^*)\to S(\h^*)$). 
It is sufficient to show that  the map $a_p^*:H^*_T(Y)\to H^*_T(T/H)$  is independent of $p\in Y$. But the map $E \times_T (T/H) \to E\times_T Y$ induced by $a_p$ 
is given by $[e, H]\mapsto [e, p]$, for all $e\in E$.
It can  be factorized as:
 \begin{equation*}
\xymatrix{ 
E\times_T(T/H)
\ar[rr]^{} 
\ar[dr]_{} 
&&E\times_TY 
 \\ 
& E\times_H Y 
\ar[ur]} \\
\end{equation*}
where $E\times_H Y \to E\times_T Y$ is the canonical map induced by the inclusion $H\subset  T$. 
The left-hand side map in the diagram is $[e, H] \mapsto[e,p]$, for all $e\in E$;
this map can also be expressed as: 
 \begin{equation*}
\vcenter{\xymatrix{
E\times_T (T/H)\ar[d]^{\cong}\ar[r]^{  } &
E\times_H Y\ar[d]^{\cong}
  &
\\
E/H \ar[r]^{ j_p: [e]\mapsto ([e], p) \ \ \ }&
(E/H) \times Y
 &
\\ 
}}
\end{equation*} 
Finally observe that $j_p^*: H^*((E/H) \times Y)\to H^*(E/H)$ is independent of $p\in Y$:
by the K\"unneth formula, $ H^*((E/H) \times Y)$ can be identified with $H^*(E/H)\otimes H^*(Y)$ and $j_p^*$ is the
projection of the latter space on $H^*(E/H)\otimes H^0(Y) \simeq H^*(E/H)$. 
\end{proof}

Observe that both $H^*_T$ and $\A_T$ are contravariant functors from the category of topological compact Hausdorff $T$-spaces to the category of graded $S(\t^*)$-algebras.
\begin{proposition}\label{only} $\gamma$ is the only natural transformation between the two functors $H^*_T$ and $\A_T$.
\end{proposition}
\begin{proof} To prove that $\gamma$ is a natural transformation we only have to verify that for every continuous $T$-equivariant map $f:X\to X'$ 
the diagram
\[
\xymatrix{
H^*_T(X') \ar[r]^{f^*} \ar[d]^{\gamma_{X'}} & H^*_T(X) \ar[d]^{\gamma_X} \\
\A_T(X') \ar[r]^{f^*} & \A_T(X)
}
\]
commutes. Take $\alpha \in H^*_T(X')$ and $p\in X$. Denote by $i_p:Tp\to X$ and $i_{f(p)}:Tf(p) \to X'$ the 
inclusion maps. Then the following diagram is commutative:
\[
\xymatrix{
Tp \ar[r]^{f|_{Tp}} \ar[d]^{i_p} & Tf(p) \ar[d]^{i_{f(p)}} \\
X \ar[r]^{f} & X'
}
\]
We thus have \begin{align*}\gamma_X(f^*(\alpha))(p)&= i_p^*(f^*(\alpha))\\{}&=(f|_{Tp})^*(i_{f(p)}^*(\alpha))=
(i_{f(p)})^*(\alpha)|_{\t_p}=\gamma_{X'}(\alpha)(f(p))|_{\t_p}=f^*(\gamma_{X'}(\alpha))(p).\end{align*}
Here we have used that $(f|_{Tp})^*:H_T^*(Tf(p))=S(\t_{f(p)}^*)\to H_T^*(Tp)=S(\t_p^*)$  is just the
restriction map induced by the inclusion $\t_p\subset \t_{f(p)}$.

For the converse we let $\eta$ denote any natural transformation between $H^*_T$ and $\A_T$. We fix an arbitrary compact Hausdorff $T$-space $X$ and show that the $S(\t^*)$-algebra homomorphism $\eta_X:H^*_T(X)\to \A_T(X)$ necessarily coincides with $\gamma_X$. As assignments are determined by their values at each point, 
we fix an arbitrary point $p$ and consider again the inclusion $i_p:T p\to X$. Then we have a commutative diagram
\[
\xymatrix{
H^*_T(X) \ar[r]^{i_p^*} \ar[d]^{\eta_X} & H^*_T(T p) \ar[d]^{\eta_{T p}} \\
\A_T(X) \ar[r]^{i_p^*} & \A_T(T p)
}
\]
Both objects on the right are isomorphic to $S(\t_p^*)$. We observe that the bottom horizontal map is just evaluation at $p$. Moreover, the vertical map on the right is necessarily the unique $S(\t^*)$-algebra homomorphism $S(\t_p^*)\to S(\t_p^*)$, namely the identity. Thus, the commutativity 
of the diagram implies that $\eta_X=\gamma_X$.
\end{proof}

\begin{remark} {\rm Observe that by restriction to the even-dimensional part of the equivariant cohomology groups, $\gamma$ induces
a transformation between the functors $H_T^{{\rm even}}$ and $\A_T$. With the same methods as in the proof of Proposition 
\ref{only}, one can show that $\gamma|_{H_T^{{\rm even}}}$ is the only natural transformation between these two functors.  
Recall that another natural transformation between  $H_T^{{\rm even}}$ and $\A_T$ was introduced in  \cite{GSZ}, in the context of smooth $T$-manifolds,  using the Cartan model of equivariant cohomology (cf.~also Section \ref{sec:smooth} above). We deduce that this transformation
coincides with  $\gamma|_{H_T^{{\rm even}}}$ on compact smooth $T$-manifolds.}
\end{remark}

As noted in \cite[Section 4]{GSZ}, if $X$ is a compact manifold and the action of $T$ on $X$ is of GKM type, then
$\gamma_X$ is an isomorphism. Here are two non-smooth examples when $\gamma_X$ is an isomorphism.

\begin{example} {\rm For the weighted projective plane ${\mathbb P}(m,n)$ already addressed in Example \ref{ex:weight}, the map
$\gamma: H^*_{S^1}({\mathbb P}(m,n))\to \A_{S^1}({\mathbb P}(m,n)) $ is an isomorphism. This follows from the GKM presentation
of   $H_{S^1}^*({\mathbb P}(m,n))$.}
\end{example}

\begin{example} \label{ex:notfree} {\rm   Consider 
$X=\Sigma T$, the unreduced suspension of the  2-torus $T=T^2$, with the canonical $T$-action.
The action has two fixed points and is free on their complement in $X$.
Both $H^*_T(X)$ and $\A_T(X)$ can be easily calculated by using the Mayer-Vietoris
sequence, see Proposition \ref{mv}. As it turns out, both algebras  are isomorphic to the space
$\{(f_1, f_2)\in S(\t^*)\times S(\t^*) \mid f_1(0) = f_2(0)\}$ (see also 
\cite[Example 5.5]{Fr-Pu}). This shows that the corresponding map $\gamma$ is again an isomorphism. 
It is shown in  \cite[Example 3.3]{All} that  $H^*_T(X)$ equipped with its canonical structure of $S(\t^*)$-module  is torsion free but not  free.
 Thus, the same can be said about $\A_T(X)$. }
\end{example}

\begin{remark}{ {\rm Observe that, in general, the assignment algebra
can be defined ``over $\bZ$": that is, in Definition \ref{defiggk}, one considers $A(p)\in S(({\mathfrak t}_p)_\bZ^*)$, 
where  $({\mathfrak t}_p)_\bZ^*$ is the weight lattice of ${\mathfrak t}_p$, $p\in X$. 
 It would be worthwhile investigating this new invariant, call it ${\mathcal A}_T(X;\bZ)$. 
 The first natural question to be addressed is to which extent the Chang-Skjelbred lemma for
 assignments, i.e.~Proposition \ref{propres} above, can be extended over $\bZ$, in the same way as
 the usual Chang-Skjelbred lemma was extended for integral cohomology by Franz and Puppe in \cite[Corollary 2.2]{FP}
 and by Goertsches and Wiemeler in \cite[Lemma 6.1]{GW}. After that, one could start searching for
 extensions of Theorem \ref{gkmassign}.  In this context, we point out that for general 
 GKM actions,   ${\mathcal A}_T(X;\bZ)$ is not isomorphic to $H^*_T(X;\bZ)$. 
 To see this, consider again Example \ref{ex:weight}, this time with $m=n=1$; that is, we just
look at the standard rotation action of $S^1$ on $S^2$, the spinning speed being $k$. Like in Example \ref{ex:weight},
${\mathcal A}_{S^1}(S^2;\bZ)$ consists of pairs 
$f_1, f_2\in \bZ[u]$ with $f_1-f_2$ divisible by $u$ (independent of $k$). However, $H^*_{S^1}(S^2;\bZ)$ does depend on $k$: it
consists of pairs $f_1, f_2\in \bZ[u]$ with $f_1-f_2$ divisible by $ku$. This example seems to indicate that, in general,
the isomrphism between the assignment algebra over $\bZ$ and the integral equivariant cohomology for GKM torus actions requires the supplementary
assumption that any of the weights at any fixed point is primitive, i.e., no integer multiple of another weight.}}
\end{remark}

\section{Locally free actions}

In this section we will assume that   $X$ is a completely regular topological space.
This assumption will allow us to use the Slice Theorem, cf., e.g.,  \cite[Ch.~II, Theorem 5.4]{Br},
which is an essential ingredient for us. For example, any Hausdorff locally compact topological space is completely regular. 
  
Consider now a subtorus $Q\subset T$, whose induced action on $X$ is locally free, i.e., the isotropy $Q_p$ is finite
for all $p\in X$.   The quotient $T/Q$ acts canonically on $X/Q$, as follows:
$tQ.Qp:=Qtp$, $t\in T, p\in X$.  In this section we show that 
$\A_T(X)$ is isomorphic to $\A_{T/Q}(X/Q)$. 
In the case when  $X$ is smooth  and the $Q$-action is smooth and free, this result 
has been proved in    \cite[Section 8]{GSZ}: the isomorphism is constructed there explicitly by relating the stratifications of
$X$ and $X/Q$. 
We adapted the approach  from the aforementioned paper to our set-up.
The main difference is that we use the pointwise definition of assignments.
The two major benefits of this definition are that the result
we will prove is purely topological, hence more general, and that the whole construction 
involves only points rather than strata, and is therefore more transparent.

Let  $\pi : X \to X/Q$ be the canonical projection. We construct a map $\pi_* : \A_T(X)\to \A_{T/Q}(X/Q)$, as follows.
Take $p\in X$. The isotropy group $(T/Q)_{Qp}$ is equal to $T(p)/Q$, where
$$T(p):=\{t\in T \mid tp  \in Qp\}.$$
The group $T(p)$ acts transitively on $Qp$, thus the latter space is homeomorphic to both
$T(p)/T_p$ and $Q/Q_p$. 

\begin{lemma}\label{thele} The map $T_p/Q_p \to T(p)/Q$ given by the inclusion of $T_p$ into $T(p)$ followed by the canonical projection
is a group isomorphism.
\end{lemma}
 
\begin{proof} The map is obviously injective. To prove surjectivity, observe that
for any $t\in T(p)$ there exists $g\in Q$ such that $tp=gp$, hence $tQ$ is the image of 
$g^{-1}tQ_p$.  
\end{proof}

Denote by  $\t_p$ and  $\t(p)$ the Lie algebras of $T_p$ and $T(p)$ respectively. 
 The differential at the identity of
the group isomorphism mentioned in Lemma \ref{thele}  is a linear isomorphism,
whose inverse is $\varphi_p : \t(p)/\q \to \t_p$.
Note that $\t(p)/\q$ is the Lie algebra of $(T/Q)_{Qp}$.  
Define $\pi_*: \A_T(X)\to \A_{T/Q}(X/Q)$, 
\begin{equation}\label{pistar}\pi_*(A)(Qp):=\varphi_p^*(A(p)), \ {\rm for \ all \ } p \in X.\end{equation}
We need to show that the map $\pi_*(A)$ satisfies the requirement of Definition  \ref{defiggk}. 
 Consider a subtorus of $T/Q$, which is of the form $H/Q$, where
$H$ is a subtorus of $T$ with $Q\subset H$. The fixed points of $H/Q$ in $X/Q$ are orbits $Qp$, with $p\in X$ such that
$Hp = Qp$.  Let $C$ be a connected component of $(X/Q)^{H/Q}$.

\begin{lemma} If $\pi : X \to X/Q$ is the canonical projection, then $\pi^{-1}(C)$ is a connected subspace of $X$. 
\end{lemma}

\begin{proof} Assume that $\pi^{-1}(C)$ is a disjoint union of two non-empty open subspaces
$U_1$ and $U_2$. Both $U_1$ and $U_2$ are $Q$-invariant: if $p\in U_1$, then
$Qp$ is connected and $Qp=(U_1\cap Qp) \cup (U_2\cap Qp)$, the elements of the union being open subspaces of $Qp$. But then $\pi(U_1)$ and $\pi(U_2)$ are disjoint as well. 
Since they are open in $C$, the latter space is not connected, which is a contradiction.
\end{proof}

Let us now denote by $\h$ and $\h_p$ the Lie algebras of $H$ and 
$H_p$, respectively, where $p\in X$. 

\begin{lemma}
The space $\h_p$ is independent of $p\in \pi^{-1}(C)$.
\end{lemma} 

\begin{proof} First, if $p \in \pi^{-1}(C)$, then $Qp=Hp$, thus 
$\dim \h_p = \dim \h - \dim \q$,   which is independent of $p$. 
From the Slice Theorem, see \cite[Ch.~II, Theorem 5.4]{Br},  any $p_0\in \pi^{-1}(C)$
has an open neighborhood where all $H$-isotropy groups are contained in $H_{p_0}$; the Lie algebras of these groups
are therefore all equal to $\h_{p_0}$. Since $\pi^{-1}(C)$ is connected and $\h_p$ is locally constant for $p \in \pi^{-1}(C)$, it is in fact globally  constant.
\end{proof}

Set $\h':=\h_p$, $p\in \pi^{-1}(C)$.  Let $H'$ be the connected subgroup of $T$ whose Lie algebra is $\h'$
(that is, the connected component of $H_p$, with $p$ as above). From the previous lemma, $\pi^{-1}(C)$ is contained in a connected component of $X^{H'}$. 
For any $p\in \pi^{-1}(C)$, the isomorphism $\varphi_p : \t(p)/\q \to \t_p$ maps $\h/\q$ to $\h \cap \t_p =\h'$.  
Since $A(p)|_{\h'}$ is independent of $p$ in the aforementioned component of $X^{H'}$, it follows that $\pi_*(A)(Qp)|_{\h/\q}$ is independent of
$Qp\in C$. 

We now have  a well defined map  $\pi_*:\A_T(X)\to \A_{T/Q}(X/Q)$. Recall that $\A_T(X)$ has a canonical
structure of an $S(\t^*)$-algebra. Via the canonical homomorphism $S((\t/\q)^*)\to S(\t^*)$, $\A_T(X)$ can
also be endowed with a structure of a $S((\t/\q)^*)$-algebra. One can verify that
$\pi_*$ is a homomorphism of $S((\t/\q)^*)$-algebras.

\begin{theorem}\label{thm:isom} The map  $\pi_*:\A_T(X)\to \A_{T/Q}(X/Q)$ is an isomorphism of $S((\t/\q)^*)$-algebras.
\end{theorem}

\begin{proof}  We show how to construct $\sigma : \A_{T/Q}(X/Q)\to \A_T(X)$ which is inverse to $\pi_*$. 
To this end we first consider for any $p\in X$ the inverse of $\varphi_p$, which is $\psi_p : \t_p \to \t(p)/\q$ 
(the inclusion of $\t_p$ into $\t(p)$ followed by the canonical projection). 
 { If $B \in \A_{T/Q}(X/Q)$, then we define} $$\sigma(B)(p):=\psi_p^*(B(Qp)), \ {\rm for \ all \ } p \in X.$$ We show that $\sigma(B)$ satisfies the requirement of  Definition \ref{defiggk}.  
Take $H'\subset T$ a subtorus with the property that $X^{H'}\neq \emptyset$. 
Then $H'\cap Q$ is a finite group. Set $H:=H'\cdot Q$ and note that its Lie algebra is $\h =\h'\oplus \q$. 
Let $Y$ be a connected component of $X^{H'}$. One can easily see that $\pi(Y)$ is contained in (a
component of) $(X/Q)^{H/Q}$.  For any $p\in Y$ we have  $\h'\subset \t_p$, hence $\psi_p(\h')=\h/\q$.  
Moreover, the map $\psi_p|_{\h'}:\h'\to \h/\q$ is independent of $p$; if we denote this map by
$\psi$, we have
$$\sigma(B)(p)|_{\h'}=\psi^*(B(Qp)|_{\h/\q}).$$ 
Since $B(Qp)|_{\h/\q}$ is constant on any connected component of $(X/Q)^{H/Q}$,
the left-hand side of the previous equation is constant on $Y$. 
At this point we conclude that the map $\sigma : \A_{T/Q}(X/Q)\to \A_T(X)$ is well defined. 
It only remains to observe that $\sigma \circ \pi_*$ and $\pi_*\circ \sigma$ are equal to the identity. 
\end{proof}

\begin{example}
{\rm Consider the action  of $T=S^1\times S^1$ on
$S^3$ described in Example \ref{sphere}. Take $S:=\{(z,z)\mid |z|=1\}$, which is  a subgroup of $T$.
It acts freely on $X$, thus $\A_T(S^3)\simeq \A_{T/S}(S^3/S).$
We have $S^3/S=\bC P^1=S^2$ and the $T/S$-action on it is equivalent to the canonical ``rotation" action of
the circle $S^1$. Along with the presentation of $\A_{S^1}(S^2)$ (see for instance \cite[Example 2.2]{GSZ}),
these identifications lead readily again to the description of $\A_T(S^3)$ given in Example \ref{sphere}. 
} 
\end{example}


\section{Kirwan surjectivity}\label{sec:kirwan}

\subsection{The assignment Kirwan map} The following set-up is mentioned in \cite[Section 8.3]{GSZ}. One considers a symplectic manifold $M$ equipped with a Hamiltonian action of a torus $T$ 
as well as a subtorus  $Q\subset T$, whose Lie
algebra  is $\q\subset \t$. The moment map of the $Q$-action is $\Phi:M\to \q^*$. 
Let $\mu \in \q^*$ be a regular value of this map. Then the action of $Q$ on the pre-image $\Phi^{-1}(\mu)$ is
locally free, hence the symplectic quotient $ M/\! /Q(\mu):=\Phi^{-1}(\mu)/Q$ has a canonical structure of a symplectic orbifold.
It also has a canonical action of the torus $T/Q$. 
One way to obtain information about the equivariant cohomology algebra $H^*_{T/Q}( M/\! /Q(\mu))$ is by
identifying it with $H^*_T(\Phi^{-1}(\mu))$; the inclusion $\Phi^{-1}(\mu) \hookrightarrow M$ induces
the algebra homomorphism
$\kappa: H^*_T(M)\to H^*_T(\Phi^{-1}(\mu))$. This is called the {\it equivariant Kirwan map} and was first studied by Goldin
in \cite{Gol}, { inspired by Kirwan's fundamental work \cite{Ki}}. Relevant for our goal is the surjectivity of this map, which holds under the assumption
that $\Phi$ is a proper map (see \cite[Theorem 1.2]{Gol}, cf.~also \cite{Ki} and \cite{To-We}).

  A natural attempt is to obtain similar results about the assignment algebra of the $T$-action on $M$.
  First,   
 by Theorem \ref{thm:isom}, $\A_T(\Phi^{-1}(\mu))\simeq \A_{T/Q}(\Phi^{-1}(\mu)/Q)$. 
 To complete the analogy with equivariant cohomology, one needs { to understand whether the map} $\kappa^A: \A_T(M) \to \A_T(\Phi^{-1}(\mu))$ 
 is surjective. We call $\kappa^A$ the  {\it assignment Kirwan map}. We first give an example 
 which shows that, in general, $\kappa^A$ is not surjective.

\begin{example}\label{ex:nonsurj} {\rm We consider the action of the torus $T^2$ on $\bC P^3$ given by
$$(e^{2\pi i t_1}, e^{2\pi i t_2}).[z_0:z_1:z_2:z_3]=[z_0: e^{4\pi i t_1}z_1: e^{4\pi i  t_2}z_2: e^{2\pi i (t_1+t_2)}z_3].$$
The canonical identification of $\t^*$ with $\t=\bR^2$ leads to the following description of a moment map:
$$\Phi : \bC P^3\to \bR^2, \quad \Phi( [z_0:z_1:z_2:z_3])=\frac{1}{|z_0|^2+|z_1|^2+|z_2|^2+|z_3|^2}(2|z_1|^2+|z_3|^2, 2|z_2|^2+|z_3|^2).$$
(As usual, $\bC P^3$ is equipped with the Fubini-Study symplectic form.) 
The circle $Q=\{(e^{2\pi i t}, e^{2\pi i t}) \mid t\in \bR\}\subset T^2$ acts on $\bC P^3$ with moment map
\begin{equation}
\label{phiq}\Phi_Q: \bC P^3 \to \bR,\quad \Phi_Q( [z_0:z_1:z_2:z_3])=2\frac{|z_1|^2+|z_2|^2+|z_3|^2}{|z_0|^2+|z_1|^2+|z_2|^2+|z_3|^2}.\end{equation}
The open subspace $U:=\{[1:z_1:z_2:z_3]\mid z_1, z_2, z_3 \in \bC\}\subset \bC P^3$ is $Q$-invariant and the moment map is the restriction
$$\Phi_Q|_U: U \to \bR,\quad \Phi_Q( [1:z_1:z_2:z_3])=2\frac{|z_1|^2+|z_2|^2+|z_3|^2}{1+|z_1|^2+|z_2|^2+|z_3|^2}.$$ 
The pre-image $\Phi_Q^{-1}(1)$ is clearly contained in $U$. 
Concretely, this space is just the unit sphere $S^5$ in $U=\bC^3$. { The induced $T^2$-action on
$S^5$ is given by 
$$(e^{2\pi i t_1}, e^{2\pi i t_2}).(z_1,z_2,z_3)=(e^{4\pi i t_1}z_1, e^{4\pi i  t_2}z_2, e^{2\pi i (t_1+t_2)}z_3),$$
for all $(z_1, z_2, z_3)\in \bC^3$ with $|z_1|^2+|z_2|^2+|z_3|^2=1$. Thus 
the  } assignment algebra  $\A_T(S^5)$ consists of triples $(f_1, f_2, f_3)$, where
\begin{align*}{}&f_1:\{(t_1, t_2)\in \bR^2\mid t_1=0\}\to \bR,\\
{}&f_2:\{(t_1, t_2)\in \bR^2\mid t_2=0\}\to \bR,\\{}&f_3:\{(t_1, t_2)\in \bR^2\mid t_1=-t_2\}\to \bR\end{align*} are polynomial functions with $f_1(0,0)=f_2(0,0)=f_3(0,0)$.
We claim that the restriction map $\A_T(U)\to \A_T(S^5)$ is not surjective
(this implies that also $\kappa^A: \A_T(\bC P^3)\to \A_T(S^5)$ is not surjective, since it factorizes by the map
above). This is because given $(f_1, f_2, f_3) \in \A_T(S^5)$, one cannot always find $f\in \bR[t_1, t_2]$
whose restrictions to the subspaces of equations $t_2=0$, $t_1=0$, and $t_1=t_2$ are $f_1, f_2, f_3$, respectively.
For instance, one can take $f_1(0,t_2)=t_2$, $f_2(t_1, 0)=t_1$, $f_3(t,-t)=t$. 
Assume there exists $f\in \bR[t_1, t_2]$ with the aforementioned properties.
We may assume that $f$ is homogeneous of degree one (otherwise, one can replace it by its degree-one component). This means that $f:\bR^2\to \bR$ is a linear map.
 It must satisfy $f(0,1)=f_1(0,1)=1$, $f(1,0)=f_2(1,0) =1$, $f(1,-1) =f_3(1,-1)=1$.
 This contradicts $f(1,-1) = f(1,0) - f(0,1)$.} \end{example}
  
  { \begin{example}\label{referee} {\rm (We are grateful to the referee for kindly suggesting this example to us.)
   The surjectivity of $\kappa^A$ does hold in the following situation.  Assume $M$ is compact, the $T$-action has finitely many
   fixed points, and the weights of the isotropy representation at any fixed point are any three linearly
   independent. Furthermore, assume that the circle $Q\subset T$ is generic, i.e.~$M^Q=M^T$, and also  that the
   $Q$-action on $\Phi^{-1}(\mu)$ is free. This setting was considered by Guillemin and Zara in \cite{GZ},
   where they proved that under the above assumptions, the $T/Q$-action on $\Phi^{-1}(\mu)/Q$ is GKM,
   see Theorem 1.5.1 in their paper.  By Theorem \ref{gkmassign} and Proposition \ref{only}, each of $M$ and $\Phi^{-1}(\mu)/Q$
   has its assignment ring and equivariant cohomology ring isomorphic to each other. 
   The surjectivity of the assignment Kirwan map then follows from the
surjectivity of the genuine (cohomological) Kirwan map.} 
     \end{example} }

\subsection{A surjectivity criterion} Let $M$ be a compact symplectic manifold acted on by a torus $T$,
the action being Hamiltonian. 
Inspired by Example \ref{phiq}, we make an assumption which 
concerns the weights of the isotropy representation at fixed points.
To formulate it, we first choose a Riemannian metric on $M$ such that $T$ acts isometrically on $M$. Let $F$ be a connected component of $M^T$. For any $p\in F$, the normal space
$\nu_pF$  
 has a complex structure which is preserved by the $T$-action. Let $\alpha_{1,F},\ldots, 
\alpha_{m,F}$ be the
  weights of the $T$-representation on $\nu_pF$ (note that they must not be pairwise distinct).
  The number $m$ is equal to half the codimension of $F$ in $M$ and it may 
  change from a connected component of $M^T$ to the other.  
The corresponding weight space decomposition is $$\nu_pF=\bigoplus_{i=1}^m\bC_{\alpha_{i,F}},$$
where $\bC_{\alpha_{i,F}}$ is a copy of $\bC$ acted on by $T$ with weight $\alpha_{i,F}$. 
We say that two functions $\alpha,\beta\in \t^*$ are equivalent, and denote this by $\alpha\sim \beta$,  
if $\alpha$ 
is a scalar multiple of $\beta$.  
The main result of this section is:

\begin{theorem}\label{surjectiv}
Assume that for any connected component $F$ of $M^T$, the elements of the set   $\{\alpha_{1,F},\ldots,\alpha_{m,F}\}/_\sim$ are linearly independent. 
Then for any circle $Q\subset T$ and any regular value $\mu $ of $\Phi:M \to \q^*$, the map $\kappa^A:\A_T(M)\to \A_T(\Phi^{-1}(\mu))$ is surjective. 
\end{theorem}

We need a preliminary result.

\begin{lemma}\label{letv} Let $V$ be a real vector space of dimension $n$.
Let also $m$ be an integer with $1\le m \le n$  and $\beta_1, \ldots, \beta_m$ some linearly independent elements
of the dual space $V^*$. Finally, let $V_1,\ldots, V_k$ be subspaces of $V$, each of them of the form
$\ker \beta_{i_1}\cap \cdots \cap \ker \beta_{i_q}$, where $i_1,\ldots, i_q\in \{1,\ldots, m\}$. 
Assume that for each $1\le i \le k$, $f_i$ is a polynomial in $S(V_i^*)$
such that $f_i|_{V_i\cap V_j}=f_j|_{V_i\cap V_j}$, for all $1\le i,j\le k$.  
Then there exists $f\in S(V^*)$ such that $f|_{V_i}=f_i$, for all $1\le i \le k$.
\end{lemma} 

\begin{proof} 
The proof is by induction on $n$. For $n=1$ the statement is trivially true.
It now follows the induction step. For any $1\le q \le m$ and  any  
$i_1, \ldots, i_q \in \{1,\ldots, m\}$ we  construct a polynomial 
$g_{i_1, \ldots, i_q}\in S((\ker \beta_{i_1}\cap \cdots \cap \ker \beta_{i_q})^*)$ such that:
\begin{itemize}
\item if $V_i = \ker \beta_{i_1}\cap \cdots \cap \ker \beta_{i_q}$
then $ g_{i_1, \ldots, i_q}=f_i$;
\item if $\{i'_1,\ldots, i'_r\} \subset \{i_1,\ldots, i_q\}$ then
$g_{i_1,\ldots, i_q} = g_{i'_1,\ldots, i'_r}|_{\ker \beta_{i_1}\cap \cdots \cap \ker \beta_{i_q}}$.
\end{itemize}

 We proceed by recursion. 
First, for $q=m$: the intersection $\ker \beta_1\cap \cdots \cap \ker \beta_m$ is equal to or contained in 
 at least one $V_i$;  we define $g_{1,\ldots, m}$ as the restriction of $f_i$ to
 $\ker \beta_1\cap \cdots \cap \ker \beta_m$. Assume that we have constructed $g$ on all
 intersections of at least $q+1$ kernels. We wish to construct $g_{i_1, \ldots, i_q}$.
 If  $\ker \beta_{i_1}\cap \cdots \cap \ker \beta_{i_q}$ is equal to $V_i$, for some $1\le i \le k$,
 we define $g_{i_1, \ldots, i_q}:=f_i$. Otherwise, we use the induction hypothesis
 to construct $g_{i_1, \ldots, i_q}$ with prescribed values on any intersection of the form
   $\ker \beta_i\cap \ker \beta_{i_1}\cap \cdots \cap \ker \beta_{i_q}$
   (note that the space of all restrictions $\beta_i|_{ \ker \beta_{i_1}\cap \cdots \cap \ker \beta_{i_q}}$
   which are not identically zero consists of linearly independent elements of  
   $(\ker \beta_{i_1}\cap \cdots \cap \ker \beta_{i_q})^*$). 

We end up with polynomials $g_1 \in S((\ker \beta_1)^*),\ldots,
g_m\in S((\ker \beta_m)^*)$ such that if $V_i =  \ker \beta_{i_1}\cap \cdots \cap \ker \beta_{i_q}$ 
is contained in $\ker \beta_j$ then $f_i = g_j |_{V_i}$. 
The goal is  to construct $f\in S(V^*)$ such that
$f|_{\ker \beta_j} = g_j$, for all $1\le j \le m$.

Set $W_j=\ker \beta_j$, $1\le j \le m$.  We can find a basis $w_1,\ldots, w_n$ of $V$ such that $W_j={\rm Span}\{w_1,\ldots, w_{j-1}, w_{j+1},\ldots, w_n\}$,
$1\le j \le m$. If $x_1,\ldots, x_n$ are the coordinates relative to this basis, then $W_j$ is described by
$x_j=0$ and $g_j$ is in $\bR[x_1,\ldots, x_{j-1}, x_{j+1}, \ldots, x_n]$.
For any $J=\{1\le j_1< \ldots < j_k\le m\}$ we denote by $J^c$ its complement in
$\{1,\ldots, n\}$; we also denote by $x_J$ the vector in
$\bR^n$ whose components are 0, except those of index $j_1, \ldots, j_k$, which  are
$x_{j_1},\ldots, x_{j_k}$ respectively. 
Set
$$f:=g_1+\cdots +g_m +\sum_{k\ge 1, I=\{1\le i_1< \ldots < i_k\le m\}}(-1)^k
g_{i_1}(x_{I^c}).$$
As one can easily see, $f(x_1,\ldots, x_{j-1}, 0, x_j,\ldots, x_n) = g_j$,
for all $1\le j \le m$.
\end{proof}

The rest of the subsection is devoted to the actual proof of Theorem \ref{surjectiv}.  
  Put an inner product on $\t$ and  identify $\q^*=\q=\bR$.
We will  use Morse theory for $f:M\to \bR$, $f(p):=(\Phi(p)-\mu)^2$, in the spirit of \cite[Ch.~10]{Ki}.
It is known, see {\it loc.~cit.}, that  $f$ is a    
  minimally degenerate function. The critical set of $f$ consists of $\Phi^{-1}(\mu)$ (minimum set) and 
$M^Q$; the connected components of the latter space are, say, $C_1, \ldots, C_N$,
such that $f(C_1) \le \ldots \le f(C_N)$. These are  critical manifolds of $f$ on $M\setminus \Phi^{-1}(\mu)$, which is a Morse-Bott
function.
 We use the standard notation $M^a:=\{p\in M \mid f(p)\le a\}$.
For any $0 < a <f(C_1)$,  there is a $T$-equivariant deformation retract of $M^a$ onto $\Phi^{-1}(\mu)$, hence, by Proposition  \ref{homotopic},
the restriction map $\A_T(M^a)\to \A_T(\Phi^{-1}(\mu))$ is surjective.

For $i\in \{1,\ldots, r\}$ we now set $C:=C_i$ and  $c:=f(C_i)$. 
To simplify the presentation, we assume that $f^{-1}(c)$ contains no other critical manifold except $C$;
 { the general case, when $f^{-1}(c)$ may contain several critical manifolds, is addressed in   Remark \ref{non-iso} below.}
 We show that for $\epsilon >0$ sufficiently small,
 the map $\A_T(M^{c+\epsilon})\to \A_T(M^{c-\epsilon})$ is surjective. 
 Let $k$ be the index of $C$ as a critical manifold of $f$. 
The negative spaces of the Hessian along $C$ give rise to a vector bundle
$E\to C$, whose rank is $k$, such that $E_q$ is a subspace of $\nu_q C$, for all $q\in C$.    
 By the equivariant Morse-Bott Lemma, { see \cite[Theorem 4.9]{W},}
$M^{c+\epsilon}$ is   $T$-equivariantly 
homotopy equivalent to the space obtained from $M^{c-\epsilon}$ by attaching the (closed) unit disk bundle $D$ in $E$
along its boundary $S$.

\begin{lemma}\label{restr} The restriction map $\A_T(D)\to \A_T(S)$ is surjective.
\end{lemma}

\begin{proof} Take $A\in \A_T(S)$. 
By Proposition \ref{homotopic}, $\A_T(D)\simeq \A_T(C)$, hence our surjectivity statement
amounts to showing that there exists $B\in \A_T(C)$ such that
for any $p\in C$ and any $v$ in the fiber $S_p$ one has
$A(v)=B(p)|_{\t_v}$. To this end, we consider the infinitesimal stratification of $C$,
whose elements are $X_1, \ldots, X_n$, with isotropy algebras $\k_1,\ldots, \k_n$ respectively,
such that if $X_a\subset \overline{X}_b$ then $b\le a$.  
For $a\in \{1,\ldots, n\}$, the weights of the (isotropy) $\k_a$-representation
on $T_pM$ are independent of $p\in X_a$.  This representation leaves both
$T_pC$ and $E_p$ invariant. Denote the weights of the $\k_a$-representation on $E_p$
by $\gamma_{a,1},\ldots, \gamma_{a, \ell}$, where $\ell$ is half the rank of $E$. 
If $F$ is a connected component of $M^T$ which is contained in $\overline{X}_a$,
the functions  $\gamma_{a,1},\ldots, \gamma_{a, \ell}$ are restrictions to $\k_a$ of certain weights
of the isotropy representation  along  $\nu F$ (more precisely, the weights of the $T$-representation
on $E_q\subset \nu_qF$, where $q\in F$). Consider
$$\k_{a,1}:=\ker \gamma_{a,1},\ldots, \k_{a, \ell}:=\ker \gamma_{a,\ell},$$
which are subspaces of $\k_a$. Note that the functions $\gamma_{a,1},\ldots, \gamma_{a, \ell}$
may be pairwise proportional or even equal and consequently the spaces above are not necessarily distinct. For any $i\in \{1,\ldots, \ell\}$ the  spaces
$\{v\in E_p \mid x.v=\gamma_{a,i}(x)v, \ {\rm for \ all \  } x\in \k_a\}$ with $p\in X_a$
give rise to a splitting of $E|_{X_a}$ as a direct sum of  $T$-equivariant subbundles.
 The vectors in the intersection of this
subbundle with  $S$  form a connected subspace of  $S$  where the
infinitesimal isotropy is  $\k_{a,i}$; hence these vectors are all mapped
to the same polynomial  $g_{a,i} \in S(\k_{a,i}^*)$. 
 The idea is to use induction on $a\in \{1,\ldots, n\}$ to construct $f_a \in S(\k_a^*)$ such that:
\begin{itemize}
\item[(i)] if $X_a\subset \overline{X}_b$ then 
$f_a|_{\k_b}=f_b$;
\item[(ii)]  $f_a|_{\k_{a,i}} =g_{a,i}$.
\end{itemize} 

(After performing this construction, 
we define $B$ as the map which assigns to each stratum $X_a$ the polynomial $f_a$.)

Let us first take $a=1$. The corresponding $X_1$ is the regular stratum of the action.
Only condition (ii) needs to be satisfied.
To justify that 
$f_1\in S(\k_1^*)$  with these properties exists, pick $F$ a component of $C^T$.
As already pointed out, 
$\gamma_{1,1},\ldots, \gamma_{1,\ell}$ are restrictions to $\k_1$ of 
some weights of the $T$-representation
along $\nu F$. But $\k_1$ is an intersection of kernels of  weights of the  same representation,
hence $\k_{1,1},\ldots, \k_{1,\ell}$ are of the same type. 
 One applies  Lemma \ref{letv} for $V=\t$ and 
$\beta_1,\ldots, \beta_m$ the  weights of the $T$-representation along $\nu F$ (modulo the equivalence relation mentioned in Theorem \ref{surjectiv}, these weights are linearly independent).  
One  uses that $A$ is an assignment on $S$. It follows that there exists a polynomial
in $S(\t^*)$ whose restriction to $\k_{1,i}$ is $g_{1,i}$, for all $i=1,\ldots, \ell$.
By restricting this polynomial to $\k_1$ one obtains the desired $f_1$.

It now follows the induction step. That is,  assuming that 
 $f_1,\ldots, f_{a-1}$ are known,  we show how to construct $f_a$. 
 First note that if $X_a \subset \overline{X}_b$ and $a\neq b$ then $b<a$ and hence $f_b$ is known.
Pick $F$ a connected component of $M^T$ contained in $\overline{X}_a$. Then $\k_a$, $\k_{a,1},\ldots, \k_{a,\ell}$ 
are all subspaces of $\t$ that can be obtained by intersecting kernels of  weights
of the $T$-representation along $\nu F$; the same can be said about $\k_b$, whenever $X_a \subset \overline{X}_b$, since $F$ is
then contained in $\overline{X}_b$. One uses again Lemma \ref{letv}. The compatibility
conditions that need to be checked  are of three types:

\begin{itemize}
\item[1.] if $X_a \subset \overline{X}_b$ and $X_a \subset \overline{X}_{b'}$
then $f_b|_{\k_b\cap \k_{b'}} = f_{b'}|_{\k_b\cap \k_{b'}}$;
\item[2.] if $X_a \subset \overline{X}_b$ and $i\in \{1,\ldots, \ell\}$, then
$f_{b}|_{\k_b\cap \k_{a,i}}=g_{a,i}|_{\k_b\cap \k_{a,i}}$;
\item[3.] if $i,i'\in \{1,\ldots, \ell\}$ then 
$g_{a,i}|_{\k_{a,i}\cap \k_{a,i'}} = g_{a,i'}|_{\k_{a,i}\cap \k_{a,i'}}$.
\end{itemize} 
To justify 1, pick again $F$ a connected component of $M^T$ contained in $\overline{X}_a$.
Pick $q\in F$ and consider the weight space decomposition of $\nu_q F$ (the normal space 
to $F$ in $C$). 
Then $\k_b\cap \k_{b'}$ is an infinitesimal isotropy of vectors/points in $\nu_qF$
that are also in a tubular neighbourhood
of $F$ in $C$. Moreover, this Lie algebra is the infinitesimal isotropy of 
a stratum, say $X_c$,  whose closure  contains $q$, as well as points in $X_b$
and points in $X_{b'}$. 
Thus $X_b \subset \overline{X}_c$ and similarly $X_{b'}\subset \overline{X}_c$.
By the induction hypothesis, both $f_b|_{\k_b\cap \k_{b'}}$ and
$f_{b'}|_{\k_b\cap \k_{b'}}$ are equal to $f_c$.
For 2, one takes into account that for any $p\in \overline{X}_b$,
the $\k_b$-representation on $E_p$ has the same weights.
If $p\in X_a$,  these weights are the restrictions to $\k_b$ of the weights of the $\k_a$-representation on $E_p$, which are $\gamma_{a,1},\ldots,\gamma_{a,\ell}$. The 
kernels of the restrictions are just $\k_b\cap \k_{a,i}$,
where $1\le i \le \ell$. 
The connected component of $C^{\k_b}$ which contains $X_b$ is a submanifold of
$C$.  For any $p$ in this submanifold one considers 
$\{v\in E_p \mid x.v=\gamma_{a,i}(x)v, 
{\rm \ for \ all \ } x\in  \k_b\}$ and obtains in this way a vector bundle.
Take $v$ in the intersection of $S$ with the fiber over $p$
and $v'$ in the intersection of $S$ with the fiber over $p'$,
where $p\in X_a$ and $p'\in X_b$. One can join $p$ and $p'$ by a 
path in $C^{\k_b}$, then one can lift it and get a path from  
$v$ to $v'$ in the vector bundle intersected with $S$. Since $A$ is an assignment on $S$,
the image of $v'$ under $A$ is  $g_{a,i}|_{\k_b\cap \k_{a,i}}$. 
Property 2 now follows from the induction hypothesis. As about 3, it  is a direct consequence of the fact that $A$ is an assignment on $S$. By Lemma \ref{letv}, conditions 1, 2 and 3
 imply that there exists a
polynomial in $S(\t^*)$ which satisfies  conditions (i) and (ii) above. 
One defines $f_a$ as the restriction to $\k_a$ of this polynomial. 
\end{proof}

Theorem \ref{surjectiv} now follows from the following lemma.

\begin{lemma} The restriction map $\A_T(M^{c+\epsilon})\to \A_T(M^{c-\epsilon})$ is surjective. \end{lemma}

\begin{proof} 
We identify $M^{c+\epsilon} =M^{c-\epsilon} \cup_S D$. The result follows readily from the Mayer-Vietoris sequence (see Proposition \ref{mv}) 
for the spaces $M^{c-\epsilon}$ and $D$, which are closed in $M^{c+\epsilon}$ and whose intersection is $S$. 
\end{proof}

{ \begin{remark}\label{non-iso} {\rm The assumption that $f^{-1}(c)$ contains only one critical manifold has been
used in the proof above. By dropping it, no essential changes are necessary. Indeed, if the critical manifolds in
$f^{-1}(c)$ are $C_{i_1},\ldots, C_{i_r}$, then, by the equivariant Morse-Bott Lemma (see \cite[Theorem 4.9]{W}),
  $M^{c+\epsilon}$ is obtained from $M^{c-\epsilon}$ by attaching some disk bundles $D_1,\ldots, D_r$ over
   $C_{i_1},\ldots, C_{i_r}$ along their boundaries $S_1, \ldots, S_r$. Like in Lemma \ref{restr}, the map
   ${\mathcal A}_T(D_i) \to {\mathcal A}_T(S_i)$ is surjective, for all $1\le i \le r$. One uses again an obvious Mayer-Vietoris argument.}   
\end{remark} } 

\begin{example} {\rm { As a direct application of Theorem \ref{surjectiv} one can show that} if $M$ equipped with a $T$-action is a toric manifold, then for any circle $Q\subset T$, the assignment Kirwan map corresponding to any  regular value $\mu$ of the 
$Q$-moment map is surjective. { In the case when $Q$ acts freely on the preimage $\Phi^{-1}(\mu)$, this is just a special case
of Example \ref{referee}.}  }
\end{example}

\begin{example} {\rm Let $T$ be an $n$-dimensional torus. Pick some  weights $\alpha_1,\ldots,\alpha_m \in \t_{\bZ}^*$
and consider first the induced actions of $T$ on $\bC P^1$, then the induced diagonal action on
$\bC P^1 \times\cdots \times \bC P^1$ ($m$ factors). The action has $2^m$ fixed points, which are $m$-tuples of the type $(p_{\pm}, \ldots, p_\pm)$, where  
$p_+=[1:0]$ and $p_-=[0:1]$.   The corresponding isotropy weights at any such point are $\pm \alpha_1, \ldots,\pm\alpha_m$.  
If the elements of $\{\alpha_1,\ldots, \alpha_m\}/_\sim$ are linearly independent, then for any circle $Q\subset T$ and
any regular value $\mu$ of the moment map $\Phi_Q$, 
the assignment moment map $\A_T(\bC P^1 \times\cdots \times \bC P^1)\to \A_T(\Phi_Q^{-1}(\mu))$ is surjective. This happens for instance if the weights are all equal,
i.e.~$\alpha_1=\ldots = \alpha_m$. Note that in this case, the rings 
$H^*_T(\bC P^1 \times\cdots \times \bC P^1)$ and $\A_T(\bC P^1 \times\cdots \times \bC P^1)$ are
not isomorphic, cf.~e.g.~\cite[Example 7.4]{GSZ}. Thus, unlike the previous example, the surjectivity of the assignment Kirwan  map is not a direct consequence of the surjectivity of the genuine Kirwan map.
}
\end{example}

\begin{remark} {\rm In Theorem \ref{surjectiv} it is essential to
make the linearly independence assumption along {\it all } connected components of
$M^T$. If one takes for instance Example \ref{ex:nonsurj}, the isotropy weights at
$[0:1:0:0]$ are $-2t_1, 2(t_2-t_1),$ and $t_2-t_1$. They are linearly independent modulo the equivalence
relation in Theorem \ref{surjectiv}. Nonetheless, we have seen that the corresponding 
$\kappa^A$ is not surjective.  { As expected, there is at least one other fixed point 
where the  assumption is not satisfied: for example, at $[1:0:0:0]$ the weights are $2t_1, 2t_2,$ and $t_1+t_2$.}
 }
 \end{remark}

\subsection{The kernel of the Kirwan map} As before,  $M$ is a compact symplectic manifold 
equipped with a Hamiltonian action of a torus $T$. 
Let $Q\subset T$ be again a circle.  Recall the identification 
$\q^* = \bR$, made by means of an inner product on $\t$. Let $\Phi : M\to \bR$ be the  moment map of the $Q$-action.  
 Under the assumption that $0$
 is a regular value of the latter map, we describe the kernel of $\kappa^A:\A_T(M) \to \A_T(\Phi^{-1}(0))$.
 Our description  is similar in spirit to the one given by Tolman and Weitsman
 \cite{To-We}  
 in the context of equivariant cohomology. 
 
 \begin{theorem}\label{kern} 
 If $0$ is a regular value of $\Phi : M \to \bR$, then the kernel of 
 $\kappa^A:\A_T(M) \to \A_T(\Phi^{-1}(0))$ is equal to the direct sum
 ${\mathcal K}^+\oplus {\mathcal K}^-$, where 
$ {\mathcal K}^{\pm}$ consist of all $A\in \A_T(M)$ with the property that
$A(F)=0$ for all connected components $F$ of $M^T$ with $\Phi(F)>0$
(resp.~$\Phi(F)<0$).
 \end{theorem}
 
 \begin{proof} We first show that if $A\in {\mathcal K}^+$ then $A(q)=0$ for all $q\in \Phi^{-1}(0)$.
 To this end, let $G$ denote the identity component of the isotropy group $T_q$
 and let $C$ be the connected component of $q$ in $M^G$. This is a $T$-invariant symplectic
 submanifold of $M$. The map $\Phi|_C$ is not constant, since 
 the action of $Q$ on $C$ is non-trivial (recall that the action of $Q$ on $\Phi^{-1}(0)$
 is locally free and $q\in \Phi^{-1}(0)$).  Observe now that $0$ is in the interior of the line segment $\Phi(C)$: otherwise $q$ would be an extremal point of $\Phi|_C$, hence a critical point,
 which is impossible, since $q$ is not $Q$-fixed (again because the action of $Q$
 on $\Phi^{-1}(0)$ is locally free). That is, $\Phi(C)$ is an interval
 $[a, b]$, where $a<0<b$. We claim that $\Phi^{-1}(b)\cap C$ contains points that are $T$-fixed.
 (The reason is that $\Phi(C)$ is obtained from the image of $C$ under  $\Phi_T:C\to \t^*$
 by projecting it orthogonally onto the line $\q^*$;
 but $\Phi_T(C)$ is a polytope whose vertices are of the form $\Phi_T(F)$,
 where $F$ is a connected component of $C^T$.)  Thus there exists $p\in C^T$ with $\Phi(p) =b>0$.
 We have $A(p)=0$ and hence  $A(q)=A(p)|_{{\rm Lie}(G)}=0$.
 
 Similarly, if $A\in {\mathcal K}^-$ then $A(q)=0$ for all $q\in \Phi^{-1}(0)$.
 We have proved that ${\mathcal K}_+\oplus {\mathcal K}_-\subset \ker \kappa^A$.
 
 The next goal is to prove the other inclusion. Take  $A$  in $\A_T(M)$ 
 whose restriction to $\Phi^{-1}(0)$ is identically $0$. 
 Consider the map $A_-$ on the set of all connected components of $M^T$ 
 with values in $S(\t^*)$ given by
 $$A_-(F) :=
 \begin{cases}
 0, \ {\rm if \ } \Phi(F)<0 \\
 A(F), \ {\rm if \ } \Phi(F) >0.
 \end{cases}
 $$ 
We show that $A_-$ extends to an assignment on $M$. 
By Theorem \ref{gkmassign}, we need to check that if $\g \subset \t$ is a
codimension-one isotropy subalgebra and $F_1$, $F_2$ 
are connected components
 of $M^T$ contained in the same connected component of $M^\g$, 
then $A_-(F_1)-A_-(F_2)$ vanishes on $\g$. This is certainly true if $\Phi(F_1)$ and
$\Phi(F_2)$ have the same sign. Let us now assume that $\Phi(F_1)<0<\Phi(F_2)$. 
The connected component of $M^\g$ mentioned above contains at least one point $q$ with
$\Phi(q)=0$. Since $q$ is not $T$-fixed, the isotropy algebra $\t_q$ is equal to $\g$.
We thus have  $A_-(F_2)|_\g=A(F_2)|_\g = A(q)=0$, which shows that
$A_-(F_1)$ is equal to $A_-(F_2)$ on $\g$.

 Similarly, take the map $A_+$ on the set of all connected components of $M^T$ 
 with values in $S(\t^*)$, given by
 $$A_+(F) :=
 \begin{cases}
 A(F), \ {\rm if \ } \Phi(F) <0\\
 0, \ {\rm if \ } \Phi(F)>0.
 \end{cases}
 $$ 
 In the same way as before, $A_+$ extends to an assignment on $M$. 
 We obviously have $A=A_++A_-$, $A_+\in {\mathcal K}^+$, and $A_-\in {\mathcal K}^-$.  
\end{proof}

\begin{example}
{\rm The building stone of our example is the ``rotation" action of $S^1$ on
 the sphere $S^2=\bC P^1$, which is
$$z.[z_1:z_2]:=[zz_1:z_2], \quad z\in S^1, [z_1:z_2]\in \bC P^1.$$
To describe a moment map it will be convenient to identify $\bC P^1$ with the
unit 2-sphere in $\bR^3$: the height function $h:\bC P^1\to \bR$ is a moment map.
The critical points are $q_+:=[1:0]$ and $q_-:=[0:1]$, the North pole and the South pole
on the sphere; that is, $h(q_+)=1$ and $h(q_-)=-1$.       
The actual example we will be looking at is the action of $T^2=S^1\times S^1$ on 
 $M:=\bC P^1\times \bC P^1 \times \bC P^1$ given by $$(z_1, z_2).(q_1, q_2, q_3) :=(z_1.q_1, z_1.q_2, z_2.q_3).$$ Let $\t^2=\bR \times \bR$ be the Lie algebra of $T^2$. A  moment map of the above action is
 $M \to \bR \times \bR$, 
 $(q_1, q_2, q_3) \mapsto (h(q_1)+h(q_2), h(q_3))$.
 Inside $T^2$ we choose the diagonal circle $\Delta (S^1) =\{(z,z)\mid z\in S^1\}$.
 By restriction to this subgroup one obtains the diagonal action of $S^1$ on $\bC P^1\times \bC P^1 \times \bC P^1$,
 whose  moment map  is $\Phi :  M \to \bR $, $\Phi(q_1,q_2,q_3)=h(q_1)+h(q_2)+h(q_3)$. 
 The critical points are the $S^1$-fixed points, that is, $(q_{\pm}, q_{\pm}, q_{\pm})$,
 eight points altogether. Consequently, the singular values are 
 $-3, -1, 1,$ and $3$. In particular, $0$ is a regular value. Denote by 
 $M_0=\Phi^{-1}(0)/\Delta(S^1)$ the symplectic quotient at $0$ and 
 set $T_0:=T^2/\Delta(S^1)$. 
 By the method described in
 this section we can calculate  $\A_{T_0}(M_0)$, as follows. First, 
 Theorem \ref{gkmassign} allows us to describe $\A_{T^2}(M)$.
 Concretely, $M^{T^2}$ consists again of the eight points $(q_{\pm}, q_{\pm}, q_{\pm})$.
 We label them as follows:
 \begin{align*}{}& p_1:=(q_-,q_-,q_-), \  p_2:=(q_+,q_-,q_-), \  p_3:=(q_-,q_+,q_-), \  p_4:=(q_-,q_-,q_+),\\
{} & p_5:=(q_+,q_+,q_-), \ p_6:=(q_+,q_-,q_+), \ p_7:=(q_-,q_+,q_+), \  p_8:=(q_+,q_+,q_+).\end{align*}
  The one-codimensional isotropies of the $T^2$-action are $\{e\}\times S^1$
 and $S^1\times \{e\}$. The fixed point sets of these two subgroups are
 $\bC P^1 \times \bC P^1 \times \{q_{\pm}\}$ and $\{q_\pm\} \times \{q_{\pm}\} \times 
 \bC P^1$.  By identifying $S((\t^2)^*)=\bR[u_1,u_2]$, we deduce that 
$\A_{T^2}(M)$ consists of all $8$-tuples $(f_1, \ldots, f_8)$ where $f_i  \in \A_{T^2}(\{p_i\}) = \bR[u_1, u_2]$, $1\le i \le 8$, 
such that:
\begin{align*} {}& f_1 - f_4, \ f_2-f_6, \ f_3 - f_7, \ {\rm and } \ f_5 - f_8 \ {\rm are \ divisible \ by} \ u_1;\\
{}& f_1 - f_2, \ f_1 - f_3, \ f_1 - f_5, \ f_4 - f_6, \ f_4 - f_7, \ {\rm and} \  f_4 - f_8  \ {\rm are \ divisible \ by} \ u_2. 
\end{align*}
A basis of $\A_{T^2}(M)$ over $\bR[u_1, u_2]$ consists of:
\begin{align*}
{}& A_1:= (1, \ 1,\ 1, \ 1,\ 1, \ 1,\ 1,\ 1);\\
{}& A_2:= (0, \ u_2, \ 0, \ 0, \ 0, \ u_2, \ 0, \ 0);\\
{}& A_3:= (0, \ 0, \ u_2,  \ 0, \ 0, \ 0, \ u_2,  \ 0);\\
{}& A_4:= (0, \ 0, \ 0, \ u_1, \ 0, \ u_1, \ u_1, \ u_1);\\
{}& A_5:= (0, \ 0, \ 0, \ 0, \ u_2, \ 0, \ 0, u_2);\\
{}& A_6:= (0, \ 0, \ 0, \ 0, \ 0, \ u_1u_2, \  0, \ 0);\\
{}& A_7:= (0, \ 0, \ 0, \ 0, \ 0,  \ 0, \ u_1u_2, \ 0);\\
{}& A_8:= (0, \ 0, \ 0, \ 0, \ 0,  \ 0, \ 0, \ u_1u_2).   
\end{align*}
 The weights of the $T^2$ isotropy action at any fixed point, regarded as vectors of
 $\t^2$, are $(\pm 1, 0), (\pm 1, 0), $ and $(0, \pm 1)$. Thus the hypotheses of
 Theorem \ref{surjectiv} are fulfilled. 
Notice now that $\Phi$ maps $p_1, p_2, p_3,$ and $p_4$ to positive numbers and 
$p_5, p_6, p_7$, and $p_8$ to negative numbers. 
Consequently, $K^+$ is spanned over $\bR[u_1, u_2]$ by:
 \begin{align*}
{}& B_1:= (u_2, \ 0,\ 0, \ u_2,\ 0, \ 0,\ 0,\ 0);\\
{}& B_2:= (0, \ u_1u_2, \ 0, \ 0, \ 0, \ 0, \ 0, \ 0);\\
{}& B_3:= (0, \ 0, \ u_1u_2,  \ 0, \ 0, \ 0, \ 0,  \ 0);\\
{}& B_4:= (0, \ 0, \ 0, \ u_1u_2, \ 0, \ 0, \ 0, \ 0);   
\end{align*}
whereas $K^-$ is spanned by $A_5, A_6, A_7$, and $A_8$.
Theorems \ref{surjectiv} and \ref{kern} imply that  $$\A_{T_0}(M_0)\simeq \A_T(M)/({\mathcal K}^+\oplus {\mathcal K}^-),$$
by an isomorphism of $S(\t_0^*)$-algebras.  Recall that $\t_0=(\bR \oplus \bR)/\Delta(\bR)$, hence
$S(\t_0^*)$ can be identified with the subring $\bR[u_1-u_2]$ of $\bR[u_1, u_2]$. 
In conclusion, $$\A_{T_0}(M_0)\simeq {\rm Span}_{\bR[u_1,u_2]}(A_1, \ldots, A_8)/{\rm Span}_{\bR[u_1,u_2]}(B_1, \ldots, B_4, A_5,\ldots, A_8),$$
the quotient in the right hand side being regarded as an $\bR[u_1-u_2]$-algebra.
A basis of $\A_{T_0}(M_0)$ as a module over $\bR[u_1-u_2]$ consists of the classes of $A_1, A_2, A_3,$ and  $A_4$.
They are a generating system since for any $i=1,2,3,4$, both $u_1A_i$ and $u_2A_i$ are linear combinations of
$(u_1-u_2)A_1$, $(u_1-u_2)A_2$, $(u_1-u_2)A_3$, $(u_1-u_2)A_4$, $B_1, B_2, B_3, B_4, A_5, A_6, A_7$, and $A_8$. 
They are also linearly independent. The details are left to the reader.

In fact, one can directly see that the canonical action of $S^1 \times S^1 \times S^1$ on $\bC P^1 \times \bC P^1 \times 
\bC P^1$ is Hamiltonian. Hence if $\Delta'(S^1):=\{(z,z,z) \mid z\in S^1\}$,
then the residual action on $M_0$ of the 2-torus $S^1\times S^1\times S^1 / \Delta'(S^1)$ is Hamiltonian as well.
The latter group contains  $T_0$ as a subgroup. By looking at the moment polytopes of the 2-torus 
and of $T_0$, one can actually see
 that  the fixed point set of the  $T_0$-action  
 has four connected components. Thus  $\A_{T_0}(M_0)$ is a free module of rank  $4$
 over $S(\t_0^*)$ also as a direct consequence of Proposition \ref{propres}, Corollary \ref{t-act}, and the fact that
 $S(\t_0^*)$ is a PID. Moreover, by Theorem 5.4, if one identifies
$S(\t_0^*)$  with $\bR[u]$, then  $\A_{T_0}(M_0)$ is isomorphic as an $\bR[u]$-algebra to the space of all $4$-tuples 
$(g_1, g_2, g_3, g_4)$ with  $g_i \in \bR[u]$  such that $g_1(0) = g_2(0) = g_3(0) = g_4(0)$.  
}
\end{example}

\end{document}